\documentclass[11pt,english]{article}
\usepackage[T1]{fontenc}
\usepackage[latin9]{inputenc}
\usepackage{mathrsfs}
\usepackage{amsmath}
\usepackage{amsthm}
\usepackage{amssymb}

\makeatletter

\newcommand{\noun}[1]{\textsc{#1}}

\theoremstyle{plain}
\newtheorem{thm}{\protect\theoremname}
  \theoremstyle{definition}
  \newtheorem*{example*}{\protect\examplename}
  \theoremstyle{plain}
  \newtheorem{prop}[thm]{\protect\propositionname}
  \theoremstyle{definition}
  \newtheorem*{problem*}{\protect\problemname}

\usepackage{mathrsfs}\usepackage{amsthm}\usepackage{amscd}

\usepackage{hyperref}
\usepackage{graphics}
\usepackage{a4wide}
\@ifundefined{definecolor}
 {\usepackage{color}}{}

\newcounter{EQNR}
\renewcommand{\contentsname}{Contents\\
{\footnotesize\normalfont(A table of contents should normally not
be included)}}

\makeatother

\usepackage{babel}
  \providecommand{\examplename}{Example}
  \providecommand{\problemname}{Problem}
  \providecommand{\propositionname}{Proposition}
\providecommand{\theoremname}{Theorem}

\begin{document}

\title{Elements of a metric spectral theory}

\author{\date{\emph{Dedicated to Margulis, with admiration}}Anders Karlsson
\thanks{Supported in part by the Swiss NSF.}}
\maketitle
\begin{abstract}
This paper discusses a general method for spectral type theorems using
metric spaces instead of vector spaces. Advantages of this approach
are that it applies to genuinely nonlinear situations and also to
random versions. Metric analogs of operator norm, spectral radius,
eigenvalue, linear functional, and weak convergence are suggested.
Applications explained include generalizations of the mean ergodic
theorem, the Wolff-Denjoy theorem and Thurston's spectral theorem
for surface homeomorphisms.
\end{abstract}

\section{Introduction}

In one line of development of mathematics, considerations progressed
from concrete functions, to vector spaces of functions, and then to
abstract vector spaces. In parallel, the standard operations, such
as derivatives and integrals, were generalized to the abstract notions
of linear operators, linear functionals and scalar products. The study
of the category of topological vector spaces and continuous linear
maps is basically what is now called functional analysis. Dieudonne
wrote that if one were to reduce the complicated history of functional
analysis to a few keywords, the emphasis should fall on the evolution
of two concepts: \emph{spectral theory} and \emph{duality} \cite{Di81}.
Needless to say, as most often is the case, the abstract general study
does not supercede the more concrete considerations in every respect.
In the context of analysis, one can compare the two different points
of view in the excellent texts \cite{L02} and \cite{StS11}. 

The metric space axioms were born out of the same development, see
the historical note in \cite{Bo87} or \cite{Di81}. In the present
paper, I would like to argue for a another step: from normed vector
spaces to metric spaces (and their generalizations), and bounded linear
operators to semicontractions. This could be called \emph{metric functional
analysis}, or in view of the particular focus here, a \emph{metric
spectral theory}. Indeed we will in the metric setting discuss a spectral
principle and duality in form of metric functionals. This is motivated
by situations which are genuinely non-linear, but there is also an
interest in the metric perspective even in the linear case. The latter
can be examplified by a well-known classical instance: for many questions
in the study of groups of $2\times2$ real matrices, it is easier
to employ their (associated) isometric action on the hyperbolic plane,
which is indeed a metric and not a linear space, instead of the linear
action on $\mathbb{R}^{2}$. The isometric action of $\textrm{PSL\ensuremath{_{2}}}(\mathbb{R})$
is by fractional linear transformations preserving the upper half
plane. This generalizes to $n\times n$ matrices and the accociated
symmetric space.

Geometric group theory is a subject that has influenced the development
of metric geometry during the last few decades. The most instrumental
contribution was made by Gromov, who in particular found inspiration
from combinatorial group theory and the Mostow-Margulis rigidity theory
(for example, the Gromov product appeared in Lyndon's work, Mostow
introduced the crucial notion of quasi-isometry, and Margulis noted
how one can argue just in terms of word metrics in this context of
quasi-isometries and boundary maps).

There is another strand of metric geometry sometimes called the \emph{Ribe
program}, see Naor's recent ICM plenary lecture \cite{N18} for some
history and appropriate references. Bourgain wrote already in 1986
\cite{B86} in this context that: ``the notions from local theory
of normed spaces are determined by the metric structure of the space
and thus have a purely metrical formulation. The next step consists
in studying these metrical concepts in general metric spaces in an
attempt to develop an analogue of the linear theory.'' The present
text suggests something similar, yet rather different. The properties
of the Banach spaces and metric spaces studied in the Ribe program
are rather subtle, in contrast here we are much more basic and in
particular motivated by understanding distance-preserving self-maps
that we see as a kind of metric spectral theory with consequences
within several areas of mathematics: geometry, topology, group theory,
ergodic theory, probability, complex analysis, operator theory, fixed
point theory etc.

We consider metric spaces $(X,d)$, at times with the symmetry axiom
removed, and the corresponding morphisms, here called semicontractions\emph{
}(in contrast to bi-Lipschitz maps in the context of Bourgain, Naor
et al)\emph{. }A map $f$ between two metric spaces is a \emph{semicontraction}
if distances are not increased, that is, for any two points $x$ and
$y$ it holds that:
\[
d(f(x),f(y))\leq d(x,y).
\]

Synonyms are $1$-Lipschitz or non-expansive maps. One could wonder
whether in such a very general setting there would be something worthwhile
to uncover, but one useful general statement is the \emph{contraction
mapping principle }which is a basic tool for finding solutions to
equations. The abstract statement appeared in Banach's thesis, but
some version might have been used before (for the existence and uniqueness
of solutions to certtain ordinary differential equations). In this
paper I will suggest a complement to this principle, which basically
appeared in \cite{Ka01}, and that is valid even more generally than
the contraction mapping principle since isometries are included.

The objective here is to discuss metric space analogs of the following
linear concepts:
\begin{itemize}
\item linear functionals and weak topology
\item operator norm and spectral radius
\item eigenvalues and Lyapunov exponents,
\end{itemize}
and then show how these metric notions can be applied. At the center
for applications is, as already indicated, a complement to the contraction
mapping principle, namely a \emph{spectral principle} \cite{Ka01,GV12},
its ergodic theoretic generalization \cite{KaM99,KaL11,GK15}, see
also \cite{G18}, and a special type of metrics that could be called
spectral metrics \cite{T86,Ka14}. 

Here is an example: Let $M$ be an oriented closed surface of genus
$g\geq2$. Let $\mathcal{S}$ denote the isotopy classes of simple
closed curves on $M$ not isotopically trivial. For a Riemannian metric
$\rho$ on $M$, let $l_{\rho}(\beta)$ be the infimum of the length
of curves isotopic to $\beta.$ In a seminal preprint from 1976 \cite{T88},
Thurston could show the following consequence (the details are worked
out in \cite[Theoreme Spectrale]{FLP79}): 
\begin{thm}
(\cite[Theorem 5]{T88}) \label{thm:Thurston-1}For any diffeomorphism
$f$ of $M$, there is a finite set $1\leq\lambda_{1}<\lambda_{2}<...<\lambda_{K}$
of algebraic integers such that for any $\alpha\in\mathcal{S}$ there
is a $\lambda_{i}$ such that for any Riemannian metric $\rho$, 
\[
\lim_{n\rightarrow\infty}l_{\rho}(f^{n}\alpha)^{1/n}=\lambda_{i}.
\]
 The map $f$ is isotopic to a pseudo-Anosov map iff $K=1$ and $\lambda_{1}>1.$
\end{thm}

This is analogous to a simple statement for linear transformations
$A$ in finite dimensions: given a vector $v$ there is an associated
exponent $\lambda$ (absolute value of an eigenvalue), such that
\[
\lim_{n\rightarrow\infty}\left\Vert A^{n}v\right\Vert ^{1/n}=\lambda.
\]
To spell out the analogy: diffeomorphism $f$ instead of a linear
transformation $A$, a length instead of a norm, and a curve $\alpha$
instead of a vector $v$. Below we will show how to get the top exponent,
even for a random product of homeomorphisms, using our metric ideas
and a lemma in Margulis' and my paper \cite{KaM99}. This is a different
appoach than \cite{Ka14}. To get all the exponents (without their
algebraic nature) requires some additional arguments, see \cite{H16}.

One of the central notions in the present text is that of a Busemann
function or metric functional. This notion appears implicitly in classical
mathematics, with Poisson and Eisenstein, and is by now receognized
by many people as a fundamental tool. In differential geometry, see
the discussion in Yau's survey \cite{Y11}. Busemann functions play
a crucial role in the Cheeger-Gromoll splitting theorem for manifolds
wiht non-degative Ricci curvature. The community of researchers of
non-posiitve curvature also has frequently emplyed Busemann functions.
For eample, it has been noted by several people that the horofunction
boundary (metric compactification) is the right notion when generalizing
Patterson-Sullivan measures, see for eample \cite{CDST18} for a recent
contribution. In my work with Ledrappier, we used this notion without
knowing anything about the geometry of the Cayley graphs, in particular
without any curvature assumption. Related to this, with a view towards
another approach to Gromov's polynomial growth theorem, see \cite{TY16}.
There are many other instances one could mention, nut still, it seems
to me that the notion of a Busmeann function remains a bit off the
mainstrem, instead of taking its natural place dual to geodeiscs. 

A note on terminology. When I had a choice, or need, to introduce
a word for a concept, I sometimes followed Serge Lang's saying that
terminology should (ideally) be functorial with respect to the ideas.
Hence I use the word \emph{metric functional} for a variant of the
the notion of horofunction usually employed and introduced by Gromov
generalizing an older concept due to Busemann in turn extending a
notion in complex analysis (and also from Martin boundary theory).
While some people do not like this, I thought it could be useful to
avoid confusion to have different terms for different concepts, even
when, or precisely because, these are variants of each other. In additon
to being functorial in the ideas, \emph{metric functional} also sounds
more basic and fundamental as a notion than \emph{horofunction} does.
Indeed, the present paper tries to argue for the analogy with the
linear case and the basic importance of the metric concept of horofunctions
or metric functionals.

For the revision of this text I thank the referee, David Fisher, Erwann
Aubry, Thomas Haettel, Massimo Picardello, Marc Peigne and especially
Armando Gutierrez for comments.

\section{Functionals}

\subsection*{Linear theory}

For vector spaces $E$, \emph{lines} 
\[
\gamma:\mathbb{R\mathrm{\rightarrow E}}
\]
are of course fundamental objects, as are their dual objects, the
\emph{linear functionals} 
\[
\phi:E\rightarrow\mathbb{R}.
\]
In the case of normed vector spaces the existence of continuous linear
functionals relies in general on Zorn's lemma via the Hahn-Banach
theorem. It is an abstraction of integrals. The sublevel sets of $\phi$
define half-spaces. The description of these functionals is an important
aspect of the theory, see for example the section entitled \emph{The
search for continuous linear functionals }in \cite{Di81}.

\subsection*{Metric theory}

For metric spaces $X$, \emph{geodesic lines} 
\[
\gamma:\mathbb{R}\rightarrow X
\]
are fundamental. The map $\gamma$ is here an isometric embedding.
(Note that \emph{geodesic} have two meanings: in differential geometry
they are locally distance minimizing, while in metric geometry they
are most often meant to be globally distance minimizing. The concepts
coincide lifted to contractible univeral covering spaces.) Now we
will discuss what the analog of linear functionals should be, that
is, some type of maps 
\[
h:X\rightarrow\mathbb{R}.
\]

\noun{Observation 1: }Let $X$ be a real Hilbert space. Take a vector
$v$ with $\left\Vert v\right\Vert =1$ and consider

\[
\lim_{t\rightarrow\infty}\left\Vert tv-y\right\Vert -\left\Vert tv\right\Vert =\lim_{t\rightarrow\infty}\sqrt{(tv-y,tv-y)}-t=\lim_{t\rightarrow\infty}\frac{(tv-y,tv-y)-t^{2}}{\sqrt{(tv-y,tv-y)}+t}=
\]

\[
=\lim_{t\rightarrow\infty}\frac{t\left(-2(y,v)+(y,y)/t\right)}{t\left(\sqrt{1-2(y,v)/t+(y,y)/t^{2}}+1\right)}=-(y,v).
\]

In this way one can recover the scalar product from the norm, differently
than from the polarization identity.

In an \emph{analytic continution of ideas }as it were, one is then
led to the next observation (which maybe is not how Busemann was thinking
about this):

\noun{Observation 2: }(Busemann) Let $\gamma$ be a geodesic line
(or just a ray $\gamma:\mathbb{R}_{+}\rightarrow X$). Then the following
limit exists:
\[
h_{\gamma}(y)=\lim_{t\rightarrow\infty}d(\gamma(t),y)-d(\gamma(t),\gamma(0)).
\]

The reason for the existence of the limit for each $y$ is that the
sequence in question is bounded from below and monotonically decreasing
(thanks to the triangle inequality), see \cite{BGS85,BrH99}. 
\begin{example*}
The open unit disk of the complex plane admits the Poincare metric,
in its infinitesimal form
\[
ds=\frac{2\left|dz\right|}{1-\left|z\right|^{2}}.
\]
This gives a model for the hyperbolic plane and moreover it is fundamental
in the way that every holomorphic self-map of the disk is a semicontraction
in this metric, this is the content of the Schwarz-Pick lemma. The
Busemann function associated to the ray from 0 to the boundary point
$\zeta$, in other words $\zeta\in\mathbb{C}$ with $\left|\zeta\right|=1$,
is
\[
h_{\zeta}(z)=\log\frac{\left|\zeta-z\right|^{2}}{1-\left|z\right|^{2}}.
\]
These functions appear (in disguise) in the Poisson integral representation
formula and in Eisenstein series. 
\end{example*}
One can take one more step, which will be parallel to the construction
of the Martin boundary in potential theory. This specific metric idea
might have come from Gromov around 1980 (except that he considers
another topology \textendash{} an important point for us here).

Let $(X,d)$ be a metric space (perhaps without the symmetric axiom
for $d$ satisfied, this point is being discussed in \cite{W11} and
\cite{GV12}). Let 
\[
\Phi:X\rightarrow\mathbb{R\mathrm{^{X}}}
\]
be defined via
\[
x\mapsto h_{x}(\cdot):=d(\cdot,x)-d(x_{0},x).
\]
This is a continuous injective map. The functions $h$ and their limits
are called \emph{metric functionals}. In view of Observation 2 Busemann
functions are examples of metric functionals and (easily seen) not
being of the form $h_{x}$, with $x\in X$. Even though geodesics
may not exist, metric functionals always exist. Note that like in
the linear case functionals are normalized to be $0$ at the origin:
$h(x_{0})=0$.

Every horofunction (i.e. uniform limit on bounded subsets of functions
$h_{x}$ as $x$tends to infinity) is a metric functional and every
Busemann function is a metric functional. On the other hand, in general
it is a well-recognized fact that not every horofunction is a Busemann
function (such spaces could perhaps be called non-reflexive) and also
not every Busemann function is a horofunction, some artifical counterexamples
showing this can be thought of: 
\begin{example*}
Take one ray $\left[0,\infty\right]$ that will be geodsic, then add
an infinite number of points at distance $1$ to the point 0 and distance
$2$ to each other. Then at each point $n$ on the ray, connect it
to one of the points around 0 with a geodesic segment of length $n-1/2$.
This way $h_{\gamma}(y)=\lim_{t\rightarrow\infty}d(\gamma(t),y)-d(\gamma(t),\gamma(0))$
still of course converge for each $y$ but not uniformly. Hence the
Busemann function $h_{\gamma}$ is a metric functional but not a horofunction.
\end{example*}
As already stated, to any geodesic ray from the origin there is an
associated metric functional (Busemann fucntion), compare this with
the situation in the linear theory that the fundamental Hahn-Banach
theorem addresses. In the metric category the theory of injective
metric spaces considers when semicontractions (1-Lipschitz maps) defined
on a subset can be extended, see \cite{La13} and references therein.
The real line is injective, which means that for any subset $A$ of
a metric space $B$ and semicontraction $f:A\rightarrow\mathbb{R}$
there is an extension of $f$ to $B\rightarrow\mathbb{R}$ without
increasing the Lipschitz constant, for example
\[
\bar{f}(b):=\sup_{a\in A}\left(f(a)-d(a,b)\right)
\]
or 
\[
\bar{f}(b):=\inf_{a\in A}\left(f(a)+d(a,b)\right),
\]
 It would require a lengthy effort to try to survey all the purposes
horofunctions have served in the past. Two instances can be found
in differential geometry, in non-negative curvature, the Cheeger-Gromoll
theorem, and in non-positive curvature, the Burger-Schroeder-Adams-Ballmann
theorem. In my experience, many people know of one or a few applications,
but few have an overview of all the applications. Other applications
are found below or in papers listed in the bibliography, for example
let us mention a recent Furstenberg-type formula for the drift of
random walks on groups \cite{CLP17} in part building on \cite{KaL06,KaL11}.
It is also the case that the last two decades have seen identifications
and understanding of horofunctions for various classes of metric spaces.

\section{Weak convergence and weak compactness}

\subsection*{Linear theory }

One of the main uses for continuous linear functionals is to define
weak topologies which have compactness properties even when the vector
space is of infinite dimension (the Banach-Aloglu theorem), see \cite{L02}.

\subsection*{Metric theory}

We will now discuss how the definition of metric functionals on a
metric space will provide the metric space with a weak topology for
which the closure is compact. There have been other more specific
efforts to achieve this in special situations, maybe the first one
for trees can be found in Margulis paper \cite{Ma81}, see also \cite{CSW93}
for another approach, \cite{Mo06} for a discussion in non-positive
curvature, and then \cite{GV12} for the general method taken here.

Let $X$ be a set. By a \emph{hemi-metric} on \emph{X} we mean be
a function 
\[
d:X\times X\rightarrow\mathbb{R}
\]
such that $d(x,y)\leq d(x,z)+d(z,y)$ for every $x,y,z\in X$ and
$d(x,y)=0$ if and only if $x=y.$ (The latter axiom can be satisfied
by passing to a quotient space.) In other words we do not insist that
$d$ is symmetric (one could symmetrize it), nor positive. For more
discussion about such metrics, see \cite{GV12,W11}. One way to proceed
is to consider 
\[
D(x,y):=\max\left\{ d(x,y),d(y,x)\right\} 
\]
which clearly is symmetric, but also positive, see \cite{GV12}, so
an honest metric. One can take the topology on $X$ from $D$.

For a weak topology there are a couple of alternative definitions,
but we proceed as follows. As defined in the previous section, let
\[
\Phi:X\rightarrow\mathbb{R\mathrm{^{X}}}
\]
defined via
\[
x\mapsto h_{x}(\cdot):=d(\cdot,x)-d(x_{0},x).
\]
This is a continuous injective map. By the triangle inequality we
note that 
\[
-d(x_{0},y)\leq h_{x}(y)\leq d(y,x_{0}).
\]
A consequence of this in view of Tychonoff's theorem is that with
the pointwise (=product) topology the closure $\overline{\Phi(X)}$
is compact. In general this is not a compactification in the strict
and standard sense that the space sits as an open dense subset in
a compact Hausdorff space, but it is convenient to still call it a
compactification, for a discussion about this terminology see \cite[6.5]{Si15}.
\begin{example*}
This has by now been studied for a number of classes of metric spaces:
non-positively curved spaces \cite{BGS85,BrH99}, Gromov hyperbolic
spaces (\cite{BrH99} or a more recent and closer to our cosideration
is \cite{MT18}), Banach spaces \cite{W07,Gu17,Gu18}, Tecihmuller
spaces (see \cite{Ka14} for references in particular to Walsh), Hilbert
metrics \cite{W11,W18,LN12}, Roller boundary of CAT(0)-cube complex
(Bader, Guralnick, Finkelshtein, unpublished), and symmetric spaces
of noncompact type equipped with Finsler metrics \cite{KL18}.
\end{example*}
Let me try to introduce some terminology. We call $\overline{\Phi(X)}$
the \emph{metric compactification }(the term was also coined for proper
geodesic metric spaces by Rieffel in a paper on operator algebras
and noncommutative geometry) and denote it by $\overline{X}$, even
though this is a bit abusive, since the topology of $X$ itself might
be different. The closure that is usually considered starting from
Gromov, see \cite{BGS85,BrH99}, is to take the topology of unform
convergence on bounded sets (note that uniform convergence on compact
sets is in the present context equivalent to our pointwise convergence),
and following \cite{BrH99} we call this the \emph{horofunction bordification}.
For proper geodesic spaces the two notions coincide.
\begin{example*}
A simple useful example is the following metric space, which I learnt
from Uri Bader. Consider longer and longer finite closed intervals
$\left[0,n\right]$ all glued to a point $x_{0}$ a the point $0.$
This becomes a countable (metric) tree which is unbounded but contain
no infinite geodesic ray. By virtue of being a tree it is CAT(0).
It is easy to directly verify that there are no limits in terms of
the topology of uniform convergence on bounded subsets. Alternatively,
one can see this less directly since for CAT(0) spaces every horofunction
is a Busemann function, but there are no (infinite) geodesic rays.
So there are no horofunctions in the usual sense, the horofunction
bordification is empty, no points are added. The metric compactification
also does not add any new points, but new topology is such that every
unbounded subsequence converges to $h_{x_{0}}$. This shows in particular
that there are minor inaccuracies in \cite[8.15 exercises]{BrH99}
and \cite[remark 14]{GV12}.
\end{example*}
Some more terminology: we call as said above the elements in $\overline{\Phi(X)}$
\emph{metric functionals}. We call \emph{horofunctions} those that
arise from unbounded subsequences via the strong topology, that of
uniform convergence on bounded subsets. The metric functionals coming
from geodesic rays, via Busemann's observation above are called \emph{Busemann
functions. }As observed above, not every Busemann function is a horofunction
and vice versa.

In my opinion these examples show the need for a precise and new terminology,
instead of just using the word \emph{horofunction} for all these concepts,
and let its precise definition depend on the context. 

Moreover, we attempt to distinguish further between various classes
of metric functionals. We have \emph{finite metric functionals} and
\emph{metric functionals at infinity. }The latter are those functions
which has $-\infty$ as its infimum. The former are hence those metric
functionals that have an finite infimum. Busemann functions are always
at infinity. The tree example above shows that even an unboudned sequence
can converge to a finite metric functional. (What can easily be shown
though is that every metric functional at infinity can only be reached
via an unbounded seuqnce). An example of a metric functional at infinity
that has finite infimum is the\emph{ $h_{\infty,0}\equiv0$} in the
Hilbert space example in the next section.

One can have metrically improper metric functionals with infinite
infimum. For the finite metric functionals we suggest moreover that
the once coming from points $x\in X$, $h_{x}$ are \emph{internal
(finite) metric functionals} and the complement of these are the \emph{exotic
(finite) metric functionals}. Examples of the latter are provided
by the Hilbert space proposition in the next section (their existence
is needed since we claim to obtain a compact space in which the Hilbert
space sits). For related division of metric functionals in the context
of Gromov hyperbolic spaces, see \cite{MT18}.
\begin{example*}
Here is a simple illustration of how the notion of metric functionals
interacts with Gromov hyperbolicity. Let $h$ be a metric functional
(Busemann function) defined by a sequence $y_{m}$ belonging to a
geodesic ray from $x_{0}$. Assume that $x_{n}$ is a sequence such
that $h(x_{n})<0$ and $x_{n}\rightarrow\infty$. Then 
\[
2\left(x_{n},y_{m}\right)=d(x_{n},x_{0})+d(y_{m},x_{0})-d(x_{n},y_{m})>d(x_{n},x_{0})
\]
for any $n$ with $m$ sufficently large in view of $0>h(x_{n})=\lim_{m\rightarrow\infty}d(y_{m},x_{n})-d(y_{m},x_{0}).$
So for each $n$ we can find a sufficently large $m$ such this inequality
holds, and along this subsequence $\left(x_{n},y_{m}\right)\rightarrow\infty$
showing that the two sequences hence converge to one and the same
point of the Gromov boundary. For more on metric functionals for (non-proper)
Gromov hyperbolic spaces we refer to \cite{MT18}.
\end{example*}

\section{Examples: Banach spaces}

\subsection*{Linear theory}

The set of continuous linear functionals form a new normed vector
space, called the\emph{ dual space, }with norm
\[
\left\Vert f\right\Vert =\sup_{v\neq0}\frac{\left|f(v)\right|}{\left\Vert v\right\Vert }.
\]

\subsection*{Metric theory}

The weak compactification and the horofunctions of Banach spaces introduces
a new take on a part of classical functional analysis, especially
as they have a similar role as continuous linear functionals. Two
features stand out, first, the existence of these new functionals
do not need any Hahn-Banach theorem which in general is based on Zorn's
lemma, second, the horofunctions are always convex and sometimes linear.
Horofunctions interpolate between the norm ($h_{0}(x)=\left\Vert x\right\Vert $)
and linear functionals. More precise statements now follow.
\begin{prop}
Let $E$ be a normed vector space. Every function $h\in\overline{E}$
is convex, that is, for any $x,y\in X$ one has
\[
h(\frac{x+y}{2})\leq\frac{1}{2}h(x)+\frac{1}{2}h(y).
\]
\end{prop}

\begin{proof}
Note that for $z\in E$ one has
\[
h_{z}((x+y)/2)=\left\Vert (x+y)/2-z\right\Vert -\left\Vert z\right\Vert =\frac{1}{2}\left\Vert x-z+y-z\right\Vert -\left\Vert z\right\Vert 
\]

\[
\leq\frac{1}{2}\left\Vert x-z\right\Vert +\frac{1}{2}\left\Vert y-z\right\Vert -\left\Vert z\right\Vert =\frac{1}{2}h_{z}(x)+\frac{1}{2}h_{z}(y).
\]
This inequality passes to any limit point of such $h_{z}$. 
\end{proof}
Furthermore, as Busemann noticed in the context of geodesic spaces,
any vector $v$ gives rise to a horofunction via
\[
h_{\infty v}(x)=\lim_{t\rightarrow\infty}\left\Vert x-tv\right\Vert -t\left\Vert v\right\Vert .
\]
Often this is a norm one linear functional, it happens precisely when
$v/\left\Vert v\right\Vert $ is a smooth point of the unit sphere
\cite{W07,Gu17,Gu18}.

Note that in this case one has in addition to the convexity that $h_{\infty v}(\lambda x)=\lambda h_{\infty v}(x)$
for scalars $\lambda,$ and so $h_{\infty v}$ is a homogeneous sublinear
function. By the Hahn-Banach theorem we have a norm 1 linear functional
$\psi$ associated to unit vector $v$ for which $\psi(v)=1$ and
such that $\psi\leq h_{\infty v}.$ 
\begin{prop}
\label{prop: Hilbert space}Let $H$ be a real Hilbert space with
scalar product $(\cdot,\cdot)$. The elements of $\overline{H}$ are
parametrized by $0<r<\infty$ and vectors $v\in H$ with $\left\Vert v\right\Vert \leq1$,
and the element corresponding to $r=0$, $v=0$. When $\left\Vert v\right\Vert =1$,
\[
h_{r,v}(y)=\left\Vert y-rv\right\Vert -r
\]
and for general $v$
\[
h_{r,v}(y)=\sqrt{\left\Vert y\right\Vert ^{2}-2(y,rv)+r^{2}}-r.
\]
In addition there is $h_{0}(y):=h_{0,0}(y)=\left\Vert y\right\Vert $
and the $r=\infty$ cases
\[
h_{\infty,v}(y)=-(y,v)
\]
where $v\in H$ with $\left\Vert v\right\Vert \mbox{\ensuremath{\leq}1.}$
A sequence $(t_{i},v_{i})$ with $\left\Vert v_{i}\right\Vert =1$
converges to $h_{r,v}$ iff $t_{i}\rightarrow r\in(0,\infty]$ and
$v_{i}\rightarrow v$ in the standard weak topology, or to $h_{0}$
iff $t_{i}\rightarrow0$.
\end{prop}

\begin{proof}
In order to identify the closure we look at vectors $tv\in H$ where
we have normalized so that $\left\Vert v\right\Vert =1$. By weak
compactness we may assume that a sequence $t_{i}v_{i}$ (or net) clusters
at some radius $r$ and some limit vector $v$ in the weak topology
with $\left\Vert v\right\Vert \leq1$. In the case $r<\infty$ we
clearly get the functions 
\[
h_{r,v}(y)=\sqrt{r^{2}(1-\left\Vert v\right\Vert ^{2})+\left\Vert y-rv\right\Vert ^{2}}-r,
\]
which after developing the norms gives the functions in the proposition.
Note that in case $t\rightarrow0$ the function is just $h_{0}$ independently
of $v$. 

In the case $t_{i}\rightarrow\infty$ we have the following calculation
\[
h_{\infty,v}(y)=\lim_{i\rightarrow\infty}\sqrt{(t_{i}v_{i}-y,t_{i}v_{i}-y)}-t=\lim_{i\rightarrow\infty}\frac{(t_{i}v_{i}-y,t_{i}v_{i}-y)-t^{2}}{\sqrt{(t_{i}v_{i}-y,t_{i}v_{i}-y)}+t}=
\]

\[
=\lim_{i\rightarrow\infty}\frac{t_{i}\left(-2(y,v)+(y,y)/t_{i}\right)}{t_{i}\left(\sqrt{1-2(y,v)/t_{i}+(y,y)/t_{i}^{2}}+1\right)}=-(y,v).
\]
It is rather immediate that the functions described are all distinct
which means that for convergent sequences both $t_{i}$ and $v_{i}$
must converge (with the trivial exception of when $t_{i}\rightarrow0$).
\end{proof}
We have in this way compactified Hilbert spaces. To illustrate the
relation with the (linear) weak topology consider an ON-basis $\left\{ e_{n}\right\} $.
It is a first example of the weak topology that $e_{n}\rightharpoonup0$
weakly. Likewise does the sequence $\lambda_{n}e_{n}$ for any sequence
of scalars $0<\lambda_{n}<1$. In $\overline{H}$ it is true that
$e_{n}\rightarrow h_{1,0}$, but $\lambda_{n}e_{n}$ does not necessarily
converge. On the other hand $n\cdot e_{1}$ does not converge weakly
as $n\rightarrow\infty$ but $n\cdot e_{1}\rightarrow h_{\infty,e_{1}}(\cdot)=-(\cdot,e_{1})$
in $\overline{H}$.

For $L^{p}$ spaces we refer to \cite{W07,Gu17,Gu18,Gu19}. An interesting
detail that Gutierrez showed is that the function identically equal
to zero is not a metric functional for $\ell^{1}$. He aslo observed
how a famous fixed point free example of Alpsach must fix a metric
functional.

\section{Basic spectral notions}

\subsection*{Linear theory}

Let $E$ be a normed vector space and $A:E\rightarrow E$ a bounded
(or continuous) linear map (operator). One defines the \emph{operator
norm}
\[
\left\Vert A\right\Vert =\sup_{v\neq0}\frac{\left\Vert Av\right\Vert }{\left\Vert v\right\Vert }.
\]
A basic notion is the\emph{ spectrum} and that it is a closed non-empty
set of complex numbers. As Beurling and Gelfand observed its radius
can be calculated by 
\[
\rho(A)=\lim_{n\rightarrow\infty}\left\Vert A^{n}\right\Vert ^{1/n}
\]
called the \emph{spectral radius} of $A$. (The existence of the limit
comes from a simple fact, known as Fekete lemma, in view of the submultiplicative
property of the norm, see \cite[17.1]{L02}). One has the obvious
inequality
\[
\rho(A)\leq\left\Vert A\right\Vert .
\]
In many important cases there is in fact an equality here, such as
for normal operators which includes all unitary and self-adjoint operators.

For a given vector $v$ one may ask for the existence of 
\[
\lim_{n\rightarrow\infty}\left\Vert A^{n}v\right\Vert ^{1/n}.
\]
Such considerations are called \emph{local spectral theory. }In infinite
dimensions this limit may not exist when the spectral theory fails.
In finite dimensions the limit exists as is clear from the Jordan
normal form. A counterexample can be given by the of $\ell^{2}$ sequence
and $A$ is a combination of a shift and a diagonal operator, having
two exponents each alternating in longer and longer stretches, making
the bahaviour seem different for various periods of n. See for example
\cite{Sc91} for details.

When $A^{n}$ is replaced by a random product of operators, an ergodic
cocycle, then Oseledets multiplicative ergodic theorem asserts that
these limits, called Lyapunov expoments, exist a.e. 

\subsection*{Metric theory}

Let $(X,d)$ be a metric space and $f:X\rightarrow X$ a semicontraction
(i.e. a 1-Lipschitz map). One defines the \emph{minimal displacement
\[
d(f)=\inf_{x}d(x,f(x)).
\]
}Like in hyperbolic geometry, or for nonpositively curved spaces \cite{BGS85},
one can classify semicontractions of a metric space as follows:
\begin{itemize}
\item \emph{elliptic} if $d(f)=0$ and the infinum is attained, i.e. there
is a fixed point
\item \emph{hyperbolic} if $d(f)>0$ and the infimum is attained, or
\item \emph{parabolic} if the minimum is not attained.
\end{itemize}
Usually the parabolic maps are the more complicated. It might also
be useful to divide semicontractions according to whether all orbits
are bounded, all orbits are unbounded, and in the latter case whether
all orbits tends to infinity. For example, a circle rotation is hyperbolic
and bounded. In this general context let me again recommend \cite{G18}
for examples and a simpler proof of Calka's theorem, which assserts
that for proper metric spaces unbounded orbits necessarly tend to
infinity. 

Another basic asociated number is the \emph{translation number} (or
\emph{drift} or \emph{escape rate})
\[
\ensuremath{\tau(f)=\lim_{n\rightarrow\infty}\frac{1}{n}d(x,f^{n}(x))}.
\]
Notice that this number is independent of $x$ because by the $1$-Lipschitz
property any two orbits stay on bounded distance from each other.
This number exists by the Fekete lemma in view of the subadditivity
coming from the triangle inequality and the 1-Lipschitz property.
It also has the tracial property: $\tau(fg)=\tau(gf)$ as is simple
to see.

One has the obvious inequality
\[
\tau(f)\leq d(f).
\]
In important cases one has equality, especially under non-positive
curvature: for isometries see \cite{BGS85} and the most general version
see \cite{GV12}. In view of that holomorphic maps preseve Kobayashi
pseudo-distances, one can study the corresponding invariants and ask
when equality holds:
\begin{problem*}
For holomorphic self-maps $f$, when do we have equality $\tau(f)=d(f)$
in the Kobayashi pseudo-distance?
\end{problem*}
This has recently been studied by Andrew Zimmer. This is analogous
to operators when the spectral radius equals the norm.

The following fact is a spectral principle \cite{Ka01} that is analogous
to the discussion about the local spectral theory. Note that in contrast
to the linear case it holds in all situations. The first statement
can also be thought of as a weak spectral theorem or weak Jordan normal
form. (For comparison, there is a stronger version in \cite{GV12}
for a restricted class of metric spaces.)
\begin{thm}
\label{thm:metric spectral}\emph{(Metric spectral principle \cite{Ka01})
}Given a semicontraction $f:(X,d)\rightarrow(X,d)$ with drift $\tau.$
Then there exists $h\in\overline{X}$ such that
\[
h(f^{k}(x_{0}))\leq-\tau k
\]
for all $k>0,$ and for any $x\in X$,
\[
\lim_{k\rightarrow\infty}-\frac{1}{k}h(f^{k}(x))=\tau.
\]
\end{thm}

\begin{proof}
Given a sequence $\epsilon_{i}\searrow0$ we set $b_{i}(n)=d(x_{0},f^{n}(x_{0}))-(l-\epsilon_{i})n.$
Since these numbers are unbounded in $n$ for each fixed $i$, we
can find a subsequence such that $b_{i}(n_{i})>b_{i}(m)$ for any
$m<n_{i}$. We have for any $k\geq1$ and $i$ that 
\[
d(f^{k}(x_{0}),f^{n_{i}}(x_{0}))-d(x_{0},f^{n_{i}}(x_{0}))
\]
\[
\leq d(x_{0},f^{n_{i}-k}x_{0})-d(x_{0},f^{n_{i}}x_{0})
\]
\[
=b_{i}(n_{i}-k)+(l-\epsilon_{i})(n_{i}-k)-b_{i}(n_{i})-(l-\epsilon_{i})n_{i}
\]
\[
\leq-(l-\epsilon_{i})k.
\]
By compactness, there is a limit point $h$ of the sequence $d(\cdot,f^{n_{i}}(x_{0}))-d(x_{0},f^{n_{i}}(x_{0}))$
in $\overline{X}$. Passing to the limit in the above inequality gives
\[
h(f^{k}(x_{0}))\leq-lk
\]
for all $k>0.$ Finally, the triangle inequality
\[
d(x,f^{k}(x))+d(f^{k}(x),z)\geq d(x,z)
\]
implies that 
\[
h(f^{k}(x_{0}))\geq-d(x_{0},f^{k}(x_{0})).
\]
From this the second statement in the theorem follows in view of that
changing $x_{0}$ to $x$ only is a bounded change since $f$ is $1$-Lipschitz:
\[
\left|d(x_{0},f^{k}(x))-d(x_{0},f^{k}(x_{0}))\right|\leq\max\left\{ d(f^{k}(x),f^{k}(x_{0})),d(f^{k}(x_{0}),f^{k}(x))\right\} 
\]
\[
\leq\max\left\{ d(x,x_{0}),d(x_{0},x)\right\} .
\]
\end{proof}
\begin{example*}
The classical instance of this is the Wolff-Denjoy theorem in complex
analysis. This is thanks to Pick's version of the Schwarz lemma which
asserts that every holomorphic map of the unit disk to itself is 1-Lipschitz
with respect to the Poincare metric $\rho$. It says that given a
holomorphic self-map of the disk, either there is a fixed point or
there is a point on the boundary circle which attracts every orbit.
From basic hyperbolic geometry one can deduce this from our theorem.
Wolff considered also horodisks, but may not have discussed lengths
$\tau,$ which here equals $\inf_{z\in D}\rho(z,f(z)),$ as folows
for example from \cite{GV12}.
\end{example*}
In the isometry case, in the same way, looking at times for which
the orbit is closer to the origin than all future orbit points, one
can show that there exists a metric functional $h$ such that 
\[
h(f^{-n}x_{0})\geq\tau_{f^{-1}}\cdot n
\]
for all $n\geq1$.

\section{Application: Extensions of the mean ergodic theorem}

In 1931 in response to a famous hypothesis in statistical mechanics,
von Neumann used spectral theory to establish that for unitary operators
$U$,
\[
\frac{1}{n}\sum_{k=0}^{n-1}U^{k}g\rightarrow Pg
\]
where $P$ is the projection operator onto the $U$ invariant elements
in the Hilbert space in question. Carleman showed this independently
at the same time (or before), and a nice proof of a more general statement
(for $U$ with $\left\Vert U\right\Vert \leq1$) was found by F. Riesz
inspired by Carleman's method. Such convergence statement is known
not to hold in general for all Banach spaces, in the sense that there
is no strong convergence of the average. On the other hand, let $f(w)=Uw+v$,
then we have 
\[
f^{n}(0)=\sum_{k=0}^{n-1}U^{k}v.
\]
If $\left\Vert U\right\Vert \leq1$, then $f$ is semi-contractive
and Theorem \ref{thm:metric spectral} applies, and it does so for
\emph{any} Banach space.

In other words the theorem is weak enough to always hold. On the other
hand when the situation is better, for example that we are studying
transformation of a Hilbert space, then the weak convergence can be
upgraded to stronger statement thanks to knowledge about the metric
functionals. Here is an example:

Let $U$ and $f$ be as above acting on a real Hilbert space. Theorem
\ref{thm:metric spectral} applied to $f$ hands us a metric functional
$h$, for which
\[
\frac{1}{n}h\left(\sum_{k=0}^{n-1}U^{k}v\right)\rightarrow-\tau,
\]
where as before $\tau$ is the growth rate in this case of the norm
of the ergodic average. Either $\tau=0$ and we have
\[
\frac{1}{n}\sum_{k=0}^{n-1}U^{k}v\rightarrow0,
\]
or else we need to have that $h$ is a metric functional at infinity
(because $h$ must be unbounded from below), see Proposition \ref{prop: Hilbert space},
in fact it must be of the form $h(x)=-(x,w)$ with $\left\Vert w\right\Vert =1$
(since $\tau$ is the growth of the norm which $h$ applied to the
orbit matches). It is a well-known simple fact that if we have a sequence
of points $x_{n}$ in a Hilbert space and a vector $w$ with norm
$\left\Vert w\right\Vert \leq1,$ such that $(x_{n},w)\rightarrow1$
and $\left\Vert x_{n}\right\Vert \rightarrow1,$ then necessarily
$x_{n}\rightarrow w$ and $\left\Vert w\right\Vert =1$. In details
for the current situation:
\[
\left\Vert \frac{1}{n}\sum_{k=0}^{n-1}U^{k}v-\tau w\right\Vert ^{2}=\left\Vert \frac{1}{n}\sum_{k=0}^{n-1}U^{k}v\right\Vert ^{2}-2\left(\frac{1}{n}\sum_{k=0}^{n-1}U^{k}v,\tau w\right)+\left\Vert \tau w\right\Vert ^{2}\rightarrow\tau^{2}-2\tau^{2}+\tau^{2}=0.
\]
This finishes the proof of the classical mean ergodic theorem.

\section{Spectral metrics}

At the moment I do not see an appropriate axiomatization for the type
of metrics that will be useful. Here is an informal description, precise
definitions will follow in the particular situations studied later.
We will have a group of transformations, with elements denoted $f$
or $g$ etc, of a space. This space has objects denoted $\alpha$
with some sort of length $l$, the set or subset of these objects
should be invariant under the transformation and we define
\[
d(f,g)=\log\sup_{\alpha}\frac{l(g^{-1}\alpha)}{l(f^{-1}\alpha)}.
\]
The triangle inequality is automatic from the supremum, as is the
invariance. The function $d$ separates $f$ and $g$ if the set of
$\alpha$s is sufficiently extensive. On the other hand this ``distance''
is not necessarily symmetric. If desired it can be symmetrized in
a couple of trivial ways.
\begin{example*}
Define a hemi-metric between two linear operators $A$ and $B$ of
a real Hilbert space $H$: 
\[
d(A,B)=\log\sup_{v\neq0}\frac{\left\Vert B^{t}v\right\Vert }{\left\Vert A^{t}v\right\Vert }.
\]
(Here $t$ denotes the transpose.) Note that we may take the supremum
over the vectors with have unit length, and also we see that there
is the obvious connection to the operator norm:
\[
d(I,A)=\log\left\Vert A^{t}\right\Vert =\log\left\Vert A\right\Vert ,
\]
where $I$ denotes the identity operator.
\end{example*}
Here is an example of classical and very useful metrics:
\begin{example*}
Metrics on the Teichmuller space of a surface,
\[
d(x,y)=\log\sup_{\alpha\in\mathscr{S}}\frac{l_{y}(\alpha)}{l_{x}(\alpha)}
\]
where $x$ and $y$ denote different equivalence classes of metrics
(or complex structures) on a fixed surface, and $\mathscr{S}$ is
the set of non-trivial isotopy classes of simple closed curves, and
$l$ could denote various notions of length, depending on the choice
the metric is asymmetric. See the next section for more details and
applications.
\end{example*}
Here is another possibility:
\begin{example*}
Taken from \cite{DKN18}. Given two intervals $I,$ $J$ and a $C^{1}$-map
$g:I\rightarrow J$ which is a diffeomorphism onto its image. The
distortion coefficient is defined by 
\[
K(g;I):=\sup_{x,y\in I}\left|\log\left(\frac{g'(x)}{g'(y)}\right)\right|.
\]
This is subadditive under composition and $K(g,I)=K(g^{-1},g(I))$.
\end{example*}
Other examples of such metrics include the Hilbert, Funk, and Thompson
metrics on cones \cite{LN12}, Kobayashi pseudo-metric in the complex
category, Hofer's metric on symplectomorphisms \cite{Gr07}. and the
Lipschitz metric on outer space. 

\section{Application: Surface homeomorphisms}

Let $\Sigma$ be a surface of finite type. Let $\mathcal{S}$ be the
set of non-trivial isotopy classes of simple closed curves on $\Sigma$.
One denotes by $l_{x}(\alpha)$ the infimal length of curves in the
class of $\alpha$in the metric $x$. The metric $x$ can be considered
to be a point in the Teichmuller space $\mathcal{T}$ of $\Sigma$
and hence a hyperbolic metric, the length will be realized on a closed
geodesic. Thurston introduced the following asymmetric metric on $\mathcal{T}$
\[
L(x,y)=\log\sup_{\alpha\in\mathcal{S}}\frac{l_{y}(\alpha)}{l_{x}(\alpha)}.
\]

Thurston in a seminal work provided a sort-of Jordan normal form for
mapping classes of diffeomorphisms of $\Sigma$ and deduced from this
the existence of Lyapunov exponents or egienvalues as it were. A different
approach was proposed in \cite{Ka14}. In this section we will use
the metrics directly, without metric functionals explicitly. We will
use a lemma in a paper by Margulis and me \cite{KaM99}, that was
substantially sharpened in \cite{GK15}. 

Let $(\Omega,\rho)$ be a measure space with $\rho(\Omega)=1$ and
let $T:\Omega\rightarrow\Omega$ be an ergodic measure preserving
map. We consider a measurable map $\omega\mapsto f_{\omega}$ where
$f_{\omega}$ are homeomorphisms of $\Sigma$ (or more generally semi-contractions
of $\mathcal{T}$). We assume the appropriate measurability and integrability
assumptions. We form $Z_{n}(\omega):=f_{\omega}\circ f_{T\omega}\circ...\circ f_{T^{n-1}\omega}$.
Let 
\[
a(n,\omega)=L(x_{0},Z_{n}(\omega)x_{0}),
\]
which is a subadditive (sub-)cocycle and by the subadditive ergodic
theorem 
\[
a(n,\omega)/n
\]
 converges for a.e. $\omega$ to a constant which we denote by $\tau.$
Given a sequence of $\epsilon_{i}$ tending to $0$, Proposition 4.2.
in \cite{KaM99} implies that a.e there is an infinite sequence of
$n_{i}$ and numbers $K_{i}$ such that
\[
a(n_{i},\omega)-a(n_{i}-k,T^{k}\omega)\geq(\tau-\epsilon_{i})k
\]
for all $K_{i}\leq k\leq n_{i}.$ Moreover we may assume that $(\tau-\epsilon_{i})n_{i}\leq a(n_{i},\omega)\leq(\tau+\epsilon_{i})n_{i}$
for all $i$.

We will now use a property of $L$ established in \cite{LRT12} (that
was not used in \cite{Ka14}). Namely there is a finite set of curves
$\mu=\mu_{x_{0}}$ such that
\[
L(x_{0},y)=\log\sup_{\alpha\in\mathcal{S}}\frac{l_{y}(\alpha)}{l_{x_{0}}(\alpha)}\asymp\log\max_{\alpha\in\mu}\frac{l_{y}(\alpha)}{l_{x_{0}}(\alpha)}
\]
up to an additive error.

Now by the pigeon-hole principle refine $n_{i}$ such that there is
one curve $\alpha_{1}$ in $\mu$ wihch realizes the maximum for each
$y=Z_{n_{i}}(\omega)x_{0}$, in other words
\[
l_{Z_{n_{i}}x_{0}}(\alpha_{1})\asymp\exp(n_{i}(\tau\pm\epsilon_{i})
\]
Given the way $n_{i}$ were selected we have
\[
-\log\sup_{\alpha\in\mathcal{S}}\frac{l_{Z_{n}x_{0}}(\alpha)}{l_{Z_{k}x_{0}}(\alpha)}\geq-a(n_{i}-k,T^{k}\omega)\geq(\tau-\epsilon_{i})k-a(n_{i},\omega)
\]
(The first inequality is an equality in case the maps are isometries,
and not merely semicontractions.) It follows, like in \cite{Ka14},
that
\[
l_{Z_{k}x_{0}}(\alpha_{1})\geq l_{Z_{n_{i}}x_{0}}(\alpha_{1})e^{-a(n_{i},\omega)}e^{(\tau-\epsilon_{i})k}.
\]
Since no length of a curve can grow faster $e^{\tau k}$ we get from
this that
\[
l_{Z_{k}x_{0}}(\alpha_{1})^{1/k}\rightarrow e^{\tau}.
\]
In other words, the top Lyapunov exponents exists in this sense. For
the other exponents in the i.i.d case we refer to Horbez \cite{H16}
and in the general ergodic setting to a forthcoming joint paper with
Horbez. The purpose of this section was to show a different technique
to such results using spectral metrics and subadditive ergodic theory.
For a similar statement instead with the complex notion of extremal
length and using metric functionals, see \cite{GK15}.

\section{Conclusion}

\subsection{A brief discussion of examples of metrics}

The hyperbolic plane, recalled above, was discovered (rather late)
as a consequence of the inquiries on the role of the parallel axiom
in Euclidean geoemtry. At that time it was probably considered a curiosiy
but later it has turned out to be a basic example, connected to an
enormous amount of mathematics. In particular it is often the first
example in the following list of metric spaces (for references see
\cite{Gr07}or \cite{Ka05,GK15}).
\begin{itemize}
\item $L^{2}$-metrics: The fundamental gorup of a Riemannian manifold acts
by isomeotry on the universal covering space. In geoemtric gorup theory
it is of importance to have isometric actions on CAT(0) spaces, for
example CAT(0)-cube complexes.
\item Symmetric space type metric spaces: Extending the role of the hyperbolic
plane for 2x2 matrices and the moduli of 2-dimensional tori, there
are the Riemannian symmetric spaces. These have recently also be considered
with Finsler metrics. Other extensions are Techmuller space, Outer
space, spaces of Riemannian metrics on which homeomorphisms or diffeomorphisms
have induced isometric actions. Likewise for invertible bounded operators
on spaces of positive operators.
\item Hyperbolic metrics: The most important notion is Gromov hyperbolic
spaces, appearing in infinite group theory (Cayley-Dehn see below),
the curve complex (non-locally compact!) and similar complexes coming
from topology and group theory, and for Hilbert and Kobayashi metrics
in the next item.
\item $L^{\infty}$-metrics. Again generalizing the hyperbolic plane and
the positivity aspect of spaces of metrics, are cones and convex sets
with metrics of Hilbert metric type. In complex analysis in one or
several variables, we have pseudo-metrics of a similar type, generalizing
the Poincare metric, the maximal one being the Kobayashi pseudo-metric.
The operator norm, Hofer's metric or Thurston's asymmetric metric
are further examples. Roughly speaking these are the metrics referred
to above as \emph{spectral metrics}, and the natiural maps in question
in all these examples are semicontractions. 
\item $L^{1}$-metrics: Cayley-Dehn graphs associated with groups and a
generating set, the group itself act on the graphs by automorphich,
which amounts to isometries with repect to the word metric.
\end{itemize}

\subsection{Further directions}

Horbez in \cite{H16} extended \cite{Ka14} to give all exponents
in the i.i.d. case, thus in particular recovering Thurston's theorem
(except for the algebraic nature of the exponents), and also implemented
the same scheme for outer automorphisms group via an intricate study
of the Culler-Vogtmann outer space, in particular its metric functionals.
Other directions could be:
\begin{itemize}
\item Symplectomorphisms and Hofer's metric
\item Reprove some statements for invertible linear transformations or compact
operators using the asymmetric metric above
\item diffeomorphisms of manifolds, there are several suggestions for spectral
metrics here. See for instance the recent preprint \cite{Na18} of
Navas on distortion of 1-dimensional diffeomorphisms. 
\end{itemize}
In the works of Cheeger and collaborators on differentiability of
functions on metric spaces, see \cite{Ch99,Ch12}, the notion of generalized
linear function appears. In \cite{Ch99} Cheeger connects this to
Busemann functions, on the other hand he remarks in \cite{Ch12} that
non-constant such functions do not exists for most spaces. Perhaps
it remains to investigate how metric functionals relate to this subject.

Section de mathematiques, Universite de Geneve, 2-4 Rue du Lievre,
Case Postale 64, 1211 Geneve 4, Suisse 

e-mail: anders.karlsson@unige.ch 

and

Matematiska institutionen, Uppsala universitet, Box 256, 751 05 Uppsala,
Sweden 

e-mail: anders.karlsson@math.uu.se
\end{document}